\documentclass[a4paper,12pt]{article}
\usepackage[utf8]{inputenc}
   \usepackage{hyperref}
   \usepackage{amsmath}
   \usepackage{amsfonts}
   \usepackage{amssymb}
   \usepackage{mathrsfs}
   \usepackage{graphicx}
   \usepackage[abs]{overpic}
   \usepackage{float}
   \usepackage{cases}
\newtheorem{theorem}{Theorem}
\newtheorem{lemma}{Lemma}
\newtheorem{corollary}{Corollary}

\newtheorem{proposition}{Proposition}

\newtheorem{remark}{Remark}

\newenvironment{proof}{\smallskip\noindent{\bf Proof}\rm}
{\hfill $\Box$\medskip}
\newcommand{\be}{\begin{equation}}
\newcommand{\ee}{\end{equation}}
\newcommand{\ba}{\begin{array}}
\newcommand{\ea}{\end{array}}
\newcommand{\bea}{\begin{eqnarray*}}
\newcommand{\eea}{\end{eqnarray*}}
\newcommand{\bean}{\begin{eqnarray}}
\newcommand{\eean}{\end{eqnarray}}

\makeatletter \@addtoreset{equation}{section}

\makeatother

\begin{document}

\title{ Eigenvalue Ratios for vibrating string equations with single-well densities }

\author{Jihed Hedhly \thanks{Facult\'e des Sciences de Tunis, Universit\'{e} El-Manar,
 Laboratoire Equations aux D\'{e}riv\'{e}es Partielles,
, jihed.hedhly@fst.utm.tn}}

\date{}

 \maketitle

\begin{abstract} In this paper, we prove the optimal upper bound
$\frac{\lambda_n}{\lambda_m}\leq(\frac{n}{m})^2$ of vibrating string
$$-y''=\lambda\rho(x) y,$$ with Dirichlet boundary conditions for single-well densities.
The proof is based on the inequality
$\frac{\lambda_n(\rho)}{\lambda_{m}(\rho)}\leq
\frac{\lambda_n(L)}{\lambda_{m}(L)} ,$ with $L$ must be a
stepfunction. We also prove the same result for the Dirichlet
Sturm-Liouville problems.

\end{abstract}
~~~~~~~~~~
\maketitle
\section{ Introduction}
We consider the Sturm-Liouville equation acting on $[0,1]$
\begin{equation} \label{1.1S}
-(p(x)y^{\prime})^ {\prime }+q(x)y=\lambda \rho (x)y,
\end{equation}
 with Dirichlet boundary
conditions
\begin{equation}
y(0)=y(1)=0  \label{1.2},
\end{equation}
where $p>0$, $\rho>0$ and $q$ (may change sign) are continuous
coefficients on $[0,1]$. Here we limit ourselves to the case
$\rho>0$. The case $\rho<0$ has been considered for related problems
providing different results, we refer to pioneering works
\cite{AC1,AC2} and some refer therein.

As is well-known (see \cite{5}), there exist two countable sequences
of eigenvalues
$$\lambda
_{1}<\lambda _{2}<\dots<\lambda _{n}\dots\infty.$$
  The issues of optimal estimates for the eigenvalue ratios
$\frac{\lambda_n}{\lambda_m}$ have attracted a lot of attention
(cf.\cite{1,2',JJ,H99,H2016,4,MH,5',M'}) and references therein.
Ashbaugh and Benguria proved in  \cite{2'} that if $q\geq0$ and
$0<k\leq p\rho(x)\leq K$, then the eigenvalues of
\eqref{1.1S}-\eqref{1.2} satisfy
%
\begin{eqnarray*}
\frac{\lambda _{n}}{\lambda _{1}}\leq \frac{Kn^{2}}{k}.
\end{eqnarray*} 
They also established the following ratio estimate (of two arbitrary
eigenvalues) $$\frac{\lambda _{n}}{\lambda _{m}}\leq
\frac{Kn^2}{km^2},\quad n> m\geq1,$$ with $q\equiv0$ and $0<k\leq
p\rho(x)\leq K$. Later, Huang and Law \cite{4} extended the results
in \cite{2'} to more general boundary conditions. \\ In the case
where $ p\equiv 1 $ and $ q \equiv 0 $, Huang proved in \cite{H99}
 that the eigenvalues for the string equation
\begin{equation} \label{1.1}
-y''=\lambda\rho(x) y,
\end{equation} with Dirichlet boundary
conditions \eqref{1.2} satisfy $\frac{\lambda_2}{\lambda_1}\leq4$
for symmetric single-well density $\rho$ and
$\frac{\lambda_2}{\lambda_1}\geq4$ for symmetric single-barrier
density  $\rho$. The later one has been extended by Horv\'{a}th
\cite{MH} for single-barrier (not necessarily symmetric) density
$\rho$. In 2006, Kiss \cite{M'} showed that
$\frac{\lambda_n}{\lambda_1}\leq n^2$ for symmetric single-well
densities and $\frac{\lambda_n}{\lambda_1}\geq n^2$ for symmetric
single-barrier densities.
  \\
 Recall that $f$ is a single-barrier (resp.
single-well) function on $[0, 1]$ if there is a point $x_{0} \in[0,
1]$ such that $f$ is increasing (resp. decreasing) on $[0, x_{0}]$
and decreasing (resp. increasing) on $[x_{0}, 1]$ (see  \cite{2}).

In this paper, we prove the optimal upper bound
$\frac{\lambda_n}{\lambda_m}\leq(\frac{n}{m})^2$
of\eqref{1.1}-\eqref{1.2} for single-well density $\rho$ (not
necessarily symmetric). The main step to prove this result is the
inequality $\frac{\lambda_n(\rho)}{\lambda_{m}(\rho)}\leq
\frac{\lambda_n(L)}{\lambda_{m}(L)} ,$ with $L$ is being a
stepfunction. We also prove an result for the Dirichlet
Sturm-Liouville problems \eqref{1.1S}-\eqref{1.2}. More precisely,
we show that $\frac{\lambda_n}{\lambda_m}\leq(\frac{n}{m})^2$ for
$q$ is single-barrier and $p\rho$ is single-well with transition
point $x_0=\frac{1}{2}$ such that
  $0<\min(\hat{\mu}_1,\tilde{\mu}_1),$
 where $\hat{\mu} _1$ and
$\tilde{\mu}_1$ are the first eigenvalues of the Neumann boundary
problems defined on $[0,\frac{1}{2}]$ and $[\frac{1}{2},1],$
respectively. \\ For this result, we modify the inverse Liouville
substitution (e.g., see  \cite[pp. 51]{M}, \cite{2'}) in order to
transform Equation \eqref{1.1S} into \eqref{1.1}, whose the density
is single-well. Therefore, we can use the result of section $2$ on
the ratio of eigenvalues $\frac{\lambda_n}{\lambda_{m}}$ for the
string equations with single-well densities.
\section{Eigenvalue ratio for the vibrating string equations }
Denote by $u_n(x)$ be the $n-th$ eigenfunction of \eqref{1.1}
corresponding to $\lambda_n$, normalized so that
$$\int_0^1\rho(x) u_n^2(x)dx=1.$$

It is well known that the $u_n(x)$ has exactly $(n-1)$ zeros  in the
open interval $(0,1)$. The zeros of the $n-th$ and $(n +1)st$
eigenfunctions interlace, i.e. between any two successive zeros of
the $n-th$ eigenfunction lies a zero of the $(n+1)st$ eigenfunction.
\\ We denote by $(y_i)_i$ the zeros of $u_n$ and $(z_i)_i$ the zeros of
$u_{n-1}$, then in view of the comparison theorem (see
\cite[Chap.1]{5}), we have $y_i<z_i$. We may assume that $u_n(x)>0$
and $u_{n-1}(x)>0$ on $(0,y_1)$, then we have
$\frac{u_n(x)}{u_{n-1}(x)}$ is strictly decreasing on $(0, 1).$ In
deed,
$$(\frac{u_n(x)}{u_{n-1}(x)})'=\frac{u'_n(x)u_{n-1}(x)-u'_{n-1}(x)u_{n}(x)}{u^2_{n-1}(x)}
=\frac{w(x)}{u^2_{n-1}(x)}.$$ We find
$$w'(x)=u''_n(x)u_{n-1}(x)-u''_{n-1}(x)u_{n}(x)=(\lambda_{n-1}-\lambda_n)\rho(x)u_n(x)u_{n-1}(x),$$
this implies that $w(x) < 0$ on $(0, 1)$. Hence
$\frac{u_n(x)}{u_{n-1}(x)}$ is strictly decreasing on $(0, 1).$ From
this, there are points $x_i\in(y_i,z_i)$ such that
\begin{eqnarray*}  \left\{
\begin{array}{ll} u^2_n(x)>u^2_{n-1}(x),\ \ \ x\in(x_{2i},x_{2i+1}),\\~~\\
 u^2_n(x)<u^2_{n-1}(x),\ \ \ x\in(x_{2i+1},x_{2i+2}).
 \end{array}
 \right.
\end{eqnarray*}

Let $\rho(., \tau)$ is a one-parameter family of piecewise
continuous densities such that $\frac{\partial\rho(.,
\tau)}{\partial\tau}$ exists, and let $u_n(x,\tau)$ be the $n-th$
eigenfunction of \eqref{1.1} corresponding to $\lambda_n(\tau)$ of
the corresponding String equation \eqref{1.1} with $\rho=\rho(.,
\tau)$.
 From Keller in \cite{KL}, we get
$$\frac{d}{d\tau}\lambda_n(\tau)=-\lambda_n(\tau)
\int_0^1\frac{\partial\rho}{\partial
\tau}(x,\tau)u_n^2(x,\tau)d\tau.$$ By straightforward computation
that, yields
\begin{equation}\label{Ra}\frac{d}{d\tau}\Big[\frac{\lambda_n(\tau)}{\lambda_{m}(\tau)}\Big]=
\frac{\lambda_n(\tau)}{\lambda_{m}(\tau)}
\int_0^1\frac{\partial\rho}{\partial
\tau}(x,\tau)(u_{m}^2(x,\tau)-u_n^2(x,\tau))d\tau.\end{equation}

We first prove.
\begin{proposition}\label{p1}
Let $\rho>0$ be monotone decreasing in $ [0,1]$ and let
$L(x)=\rho(x_{2i+1})$ (where $ x_i $ the points such that $ u_n ^ 2
(x_i) = u_ {n-1}^2 (x_i) )$, then
\begin{equation}\label{com}
   \frac{\lambda_{n}(\rho)}{\lambda_{m}(\rho)}\leq
   \frac{\lambda_{n}(L)}{\lambda_{m}(L)}.
   \end{equation}
   with equality if and only if $\rho\equiv L$.
\end{proposition}
\begin{proof} Define $\hat{\rho}(x,\tau)=\tau\rho(x)+(1-\tau)L(x).$
Using \eqref{Ra}, one gets
\begin{eqnarray} \label{A.C}
&&\frac{d}{d\tau}\Big[\frac{\lambda_n(\tau)}{\lambda_{n-1}(\tau)}\Big]=
\frac{\lambda_n(\tau)}{\lambda_{n-1}(\tau)}
\int_0^1\frac{\partial\hat{\rho}}{\partial
\tau}(x,\tau)(u_{n-1}^2(x,\tau)-u_n^2(x,\tau))d\tau \cr
&&=\frac{\lambda_n(\tau)}{\lambda_{n-1}(\tau)}\sum_{i=0}^n
\int_{x_{2i}}^{x_{2i+2}}(\rho(x)-L(x))(u_{n-1}^2(x,\tau)-u_n^2(x,\tau))d\tau.
\end{eqnarray}
We notice that
$$\int_{x_{2i}}^{x_{2i+2}}(\rho(x)-L(x)(u_{n-1}^2(x,\tau)-u_n^2(x,\tau))d\tau\leq0.$$
It then follows that
$\frac{d}{d\tau}\Big[\frac{\lambda_n(\tau)}{\lambda_{n-1}(\tau)}\Big]\leq0$.
Thus, by the continuity of eigenvalues, we obtain
$$\frac{\lambda_n(\rho)}{\lambda_{n-1}(\rho)}=\frac{\lambda_n(1)}{\lambda_{n-1}(1)}\leq
\frac{\lambda_n(0)}{\lambda_{n-1}(0)}=\frac{\lambda_n(L)}{\lambda_{n-1}(L)}
.$$ And hence
$$\frac{\lambda_n(\rho)}{\lambda_{m}(\rho)}\leq
\frac{\lambda_n(L)}{\lambda_{m}(L)} .$$ Equality holds iff $\rho=L$.
\end{proof}\\
 We are now in position to state our main result.
\begin{theorem}\label{theo11}
Let $\rho$ be a single-well density on $[0, 1]$. Then the
eigenvalues of the Dirichlet problem \eqref{1.1}-\eqref{1.2} satisfy
   \begin{equation}\label{RA}
   \frac{\lambda_{n}}{\lambda_{m}}\leq(\frac{n}{m})^2,
   \end{equation}
with equality if and only if $\rho$ is constant.
 \end{theorem}
 The proof of Theorem \ref{theo11} will be given in section $ 3$.

\section{Proof of Theorems \ref{theo11}}

\begin{corollary}\label{c2}
Consider equation \eqref{1.1} with the Dirichlet boundary conditions
\eqref{1.2}. If the density $\rho$ is decreasing in $[0,1]$, then
the $m-th$ and $n-th$ eigenvalues with $m < n$ satisfy
\begin{equation*}
   \frac{\lambda_{n}}{\lambda_{m}}\leq(\frac{n}{m})^2.
   \end{equation*} Equality holds iff $\rho$ is constant.
\end{corollary}
In order to prove Corollary \ref{c2} we need some preliminary
results.\\ Let $y( x,z ) $ be the unique solution of the initial
value problem
\begin{eqnarray} \label{H.S} \left\{
\begin{array}{ll}-y^{\prime \prime }=z^{2}\rho(x)y,\ \ \ x\in [ 0,1] ,\ \ \ z>0,\\
y(0)=0, \ \ \  y^{\prime }(0)=\rho^{\frac{1}{4}}(0).
 \end{array}
 \right.
\end{eqnarray}
We shall apply to System \eqref{H.S}, the modified Pr\"{u}fer
substitution as introduced in \cite{M'}.
\begin{eqnarray}\label{2.2}
&&y(x,z) =\frac{r( x,z)}{z}\rho^{\frac{-1}{4}} \sin \varphi ( x,z),
\cr &&y^{\prime }( x,z)=r(x,z)\rho^{\frac{1}{4}} \cos \varphi (
x,z), \cr &&\varphi (0,z) =0,
\end{eqnarray}
where $r(x,z)>0$, and then let $\theta(x,z) =\frac{\varphi
(x,z)}{z}.$ We denote by prime $({resp.~ dot})$ the derivative with
respect to $x$ $({resp. ~z})$.\\ Using Equation $\eqref{1.1} $
together with \eqref{2.2}, one finds the following one finds the
following differential equations for $r( x,z)$ and $\varphi(x,z)$:
\begin{equation} \label{2.6}
\varphi ^{\prime
}=z\rho^{\frac{1}{2}}+\frac{1}{4}\frac{\rho'}{\rho}\sin (2\varphi ),
\end{equation}%
\begin{equation}\label{2.7}
\frac{r^{\prime }}{r}=\frac{-1}{4}\frac{\rho'}{\rho} \cos(2 \varphi)
.
\end{equation}%
\begin{lemma}\label{lem1}
\begin{equation}\label{2.9}
\dot{\varphi}=\int_0^x\rho^{\frac{1}{2}}(t)\frac{r^2(t,z)}{r^2(x,z)}dt.
\end{equation}
\end{lemma}
\begin{proof} Differentiate equation \eqref{2.6} with respect to
$z:$

\begin{equation}\label{2.9v}
\dot{\varphi'}=\rho^{\frac{1}{2}}+\frac{1}{4}2\varphi\frac{\rho'}{\rho}\cos
(2\varphi ). \end{equation} Multiplying both sides by
$e^{\int_0^x\frac{r'(t)}{r(t)}dt}$, yields
$$\dot{\varphi}=\int_0^x\rho^{\frac{1}{2}}(t)\frac{r^2(t,z)}{r^2(x,z)}dt.$$
\end{proof}
\begin{corollary}\label{c1}
\begin{equation}\label{2.9t}
\dot{\theta}( x,z) =\frac{1}{z^{2}r^{2}(x)}\int_{0}^{x}r^{2}(t)
\Big[2z\rho^{\frac{1}{2}}(t)+\frac{1}{4}\frac{\rho'(t)}{\rho(t)}\Big(
\sin(2\varphi(t))+2\varphi(t)\cos (2\varphi) \Big)\Big] dt.
\end{equation}
\end{corollary}
\begin{proof}
\begin{eqnarray*}
&&\dot{\theta}( x,z)=\frac{\dot{\varphi}(x,z)}{z}-\frac{
\varphi(x,z)}{z^2}\cr&&=\frac{1}{z}\int_0^x\rho^{\frac{1}{2}}(t)\frac{r^2(t,z)}{r^2(x,z)}dt-\frac{
\varphi(x,z)}{z^2}\cr&&=\frac{1}{z^2r^2(x,z)}\Big[
\int_0^xz\rho^{\frac{1}{2}}(t) r^2(t,z)dt- r^2(x,z)\varphi(x,z)\Big]
\cr&&=\frac{1}{z^2r^2(x,z)}\Big[
\int_0^xr^2(t,z)z\rho^{\frac{1}{2}}(t)dt-
2\int_0^xr(t)r'(t)\varphi(t,z)dt+\int_0^xr^2(t)\varphi'(t,z)dt\Big]\cr&&=
\frac{1}{z^2r^2(x,z)}\Big[
\int_0^xr^2(t,z)(2z\rho^{\frac{1}{2}}(t)+\varphi'(t,z)) dt-
2\int_0^xr(t)r'(t)\varphi(t,z)dt\Big]\cr&&=
\frac{1}{z^2r^2(x,z)}\Big[\int_0^xr^2(t,z)(2z\rho^{\frac{1}{2}}(t)+\varphi'(t,z))
dt-2\int_0^xr^2(t,z)\frac{r'(t,z)}{r(t,z)}\varphi(t,z)dt\Big]
\cr&&=\frac{1}{z^2r^2(x,z)}\Big[\int_0^xr^2(t,z)(2z\rho^{\frac{1}{2}}(t)+\varphi'(t,z))
dt+\frac{2}{4}\int_0^xr^2(t,z)\frac{\rho'(t,z)}{\rho(t,z)}\varphi(t,z)\cos(2\varphi)dt\Big]
\cr&&=\frac{1}{z^{2}r^{2}(x)}\int_{0}^{x}r^{2}(t)
\Big[2z\rho^{\frac{1}{2}}(t)+\frac{1}{4}\frac{\rho'(t)}{\rho(t)}\Big(
\sin(2\varphi(t))+2\varphi(t)\cos (2\varphi) \Big)\Big] dt.
\end{eqnarray*}
\end{proof}

We can now prove Corollary \ref{c2}.

\begin{proof}
Let $L(x)=\rho(x_{2i+1}),$ for all $x\in(x_{2i},x_{2i+2}) ,$ then
$L'\equiv0$ for all $x\in(x_{2i},x_{2i+2}).$ Using Corollary
\ref{c1}, we obtain

\begin{eqnarray*}&& \dot{\theta}( x,z)
=\frac{2}{zr^{2}(x)}\int_{0}^{x}r^{2}(t)L^{\frac{1}{2}}(t) dt
 \geq0.
\end{eqnarray*}
Therefore, $ \dot{\theta}( x ,z) \geq 0 .$ Let $m$ be less than $n$.
Then $\frac{m\pi}{z_m}=\theta(z_m)\leq\frac{n\pi}{z_n}=\theta(z_n)$,
and thus $\frac{z_n}{ z_m} \leq \frac{n}{ m}$ and
$\frac{\lambda_n(L)}{ \lambda_m(L)} \leq (\frac{n}{ m})^2.$ Then
from Proposition \ref{p1}, we get
$$\frac{\lambda_n }{ \lambda_m } \leq (\frac{n}{ m})^2.$$
 The equality iff $\frac{\lambda_n(t)}{
\lambda_m(t)} $ is a constant. From \eqref{2.7}, $ \dot{\theta}( x
,z) = 0 $, which implies that $\rho\equiv\hat{\rho}\equiv cte$. This
completes the proof of the theorem.
\end{proof}\\

\begin{proof} of Theorem
\ref{theo11}\\ We define $\tilde{\rho}(x)=\rho(1-x),$ then
$\tilde{\rho}(x)$ is monotone decreasing in $[0,1-x_0]$ and monotone
increasing in $[1-x_0,1].$ According to Proposition \ref{p1}
together with Corollary \ref{c2}, yields
$$\frac{\lambda_n }{ \lambda_m}=\frac{\lambda_n(\rho) }{ \lambda_m (\rho)}\leq
\frac{\lambda_n(L) }{ \lambda_m (L)} \leq \big(\frac{n}{
m}\big)^2.$$ The equality holds, if $\rho$ is a constant.
\end{proof}

\section{Eigenvalue ratios for Sturm-Liouville problems}
In this section, For this result, we modify the inverse Liouville
substitution (e.g., see  \cite[pp. 51]{M}, \cite{2'}) in order to
transform Equation \eqref{1.1S} into \eqref{1.1}, whose the density
is single-well. Therefore, we can use the result of section $2$ on
the ratio of eigenvalues $\frac{\lambda_n}{\lambda_{m}}$ for the
string equations with single-well densities.\\
\begin{theorem}\label{theo13}
Let $q$ be a single-barrier potential and
 $p\rho$ be a single-well function with transition point $x_0=\frac{1}{2}$ such that
  $0<\min(\hat{\mu}_1,\tilde{\mu}_1)$
 where $\hat{\mu} _1$ and
$\tilde{\mu}_1$ are the first eigenvalues of the Neumann boundary
problems defined on $[0,\frac{1}{2}]$ and $[\frac{1}{2},1],$
respectively. Then the eigenvalues of Problem
\eqref{1.1S}-\eqref{1.2} satisfy
   \begin{eqnarray}\label{12}
 \frac{\lambda_{n}}{\lambda_{m}}\leq (\frac{n}{m})^2.
 \end{eqnarray}
 Equality holds iff $q\equiv0$ and $p\rho$ is constant in $[0,1]$.
 \end{theorem}
 The following result is stated without assumptions on the
monotonicity on the potential $q.$
 \begin{corollary}\label{c3}
 If $q$ is nonnegative and $p\rho$ is single-well with transition point $x_0=\frac{1}{2}$, then
\begin{eqnarray*}\label{NOS}\frac{\lambda_{n}}{\lambda_{m}}\leq (\frac{n}{m})^2.\end{eqnarray*}
 Equality holds iff $q\equiv0$ and $p\rho$ is constant on
 $[0,1]$.

 \end{corollary}
 For the proof of Theorem \eqref{theo13}, we need some preliminary
 results.

 Let $h(x,\lambda)$ be the unique solution of Equation \eqref{1.1S}
satisfying the initial conditions
\begin{eqnarray}\label{H1}
h(1/2)=1,~ h'(1/2)=0.
\end{eqnarray} We introduce the meromorphic
function
\begin{eqnarray}\label{F}
F(x,\lambda)=\frac{ph'(x,\lambda)}{h(x,\lambda)}.
\end{eqnarray}
Let $\hat{\eta}_1$ and $\tilde{\eta}_1$ be the first eigenvalues of
the problems determined by Equation \eqref{1.1S} and the boundary
conditions
        \begin{eqnarray}\label{hat}
 y(0)=y'(1/2)=0,
\end{eqnarray}
        \begin{eqnarray}\label{tilde}
y'(1/2)=y(1)=0,
 \end{eqnarray} respectively.
 \begin{lemma} \label{lem11} ~~\\
 \begin{itemize}
          \item The function $F(0,\lambda)$
          is increasing 
along the interval $(-\infty,\hat{\eta}_1)$.
\item The function $F(1,\lambda)$ is decreasing
          along the interval
          $(-\infty,\tilde{\eta}_1)$.
        \end{itemize}
\end{lemma}
\begin{proof}
The proof is similar to that of Lemma $3.3$ in \cite{JJ}.
\end{proof}

We are now ready to prove Theorem \ref{theo13}.

\begin{proof} Firstly, if
$q\equiv0$ then by use the Legendre substitution \cite[pp.
227-228]{L}
\begin{eqnarray}\label{SB}
t(x)=\frac{1}{\sigma}\int_{0}^{x}\frac{1}{p(z)}dz,~~
\sigma=\int_{0}^{1}\frac{1}{p(z)}dz,\end{eqnarray} Equation
\eqref{1.1S} can be rewritten in the string equation
$$-\ddot{y}=\lambda \sigma^2\tilde{p}(t)\tilde{\rho}(t)y,$$ where
$\tilde{p}(t)=p(x)$ and $\tilde{\rho}(t)=\rho(x)$. Thus the estimate
\eqref{12} is direct consequence of Theorem \ref{theo11}. In the
sequel we suppose that $q\not\equiv0$. Assume that
$\hat{\mu}_1=\min(\hat{\mu}_1,\tilde{\mu}_1)$ and let $h$ be the
unique solution of the second-order equation
\begin{eqnarray}\label{0'} (p(x)y')'=q(x)y ,\end{eqnarray} satisfying the initial
conditions \eqref{H1}. Hence, by the hypothesis and the variational
principle, $$0<\hat{\mu}_1<\hat{\eta}_1.$$
 It is known that $h(x,\hat{\eta}_1)>0$ on
$(0,\frac{1}{2}]$, by Sturm comparison theorem (see \cite[Chap.
1]{5}), we have $h(x)>0$ on $[0,\frac{1}{2}].$\\ On the other hand,
since $F(0,\hat{\mu}_1)=0$, then in view of Lemma \ref{lem11},
$F(0,\lambda)\leq0$ on $(-\infty,\hat{\mu}_1]$. Hence, from the
condition $0<\hat{\mu}_1,$ together with $h(0)>0$, we get
$h'(0)\leq0.$ \\ Taking into account that $q$ is increasing on
$[0,\frac{1}{2}]$, then it may vanish at most once, say at
$a_0\in[0,\frac{1}{2})$. From this and \eqref{0'}, we have
$(ph')'\leq0$ on $[0,a_0]$ and $(ph')'\geq0$ on $[a_0,\frac{1}{2}]$,
and consequently $ph'$ is decreasing on $[0,a_0]$ and increasing on
$[a_0,\frac{1}{2}]$. Since $h'(\frac{1}{2})=0$ and $h'(0)\leq0$,
then $h'(x)\leq0$ on $[0,\frac{1}{2}].$ Using similar arguments, it
can be shown that $h'(x)\geq0$ on $[\frac{1}{2},1].$ Therefore, $h$
is a single-well function on $[0,1]$. We introduce the modified
inverse Liouville substitution (e.g., see \cite[pp. 51]{M},
\cite{2'})
\begin{eqnarray}\label{inv}
z(x)=\frac{1}{c}\int_{0}^{x}\frac{1}{h^2(s)}ds,~where~c=\int_{0}^{1}\frac{1}{h^2(s)}ds,
\end{eqnarray}
which transforms Problem  \eqref{1.1S}-\eqref{1.2} into the system
\begin{eqnarray} \label{3.5}
\left\{
 \begin{array}{ll} -(p(z)u^\cdot)^{\cdot}=\tilde{\lambda}
 \tilde{h}^4(z)\rho(z) u,
 ~z\in(0,1),\\
u(0)=u(1)=0,
 \end{array}
 \right.
\end{eqnarray} where $y=uh,$ $\tilde{h}(z)=h(x)$ and
$\tilde{\lambda}=c^2\lambda$. Using the Legendre substitution
\eqref{SB}, the System \eqref{3.5}, becomes a string equation
\begin{eqnarray}\label{1u1}
-\ddot{u}=\sigma^2\tilde{\lambda}
\hat{h}^4(t)\hat{p}(t)\hat{\rho}(t) u,~t\in(0,1),
\end{eqnarray}
 with Dirichlet boundary conditions
\begin{equation}
u(0)=u(1)=0, \label{1.2'1}
\end{equation} where $\hat{p}(t)=p(x),~\hat{\rho}(t)=\rho(x)$ and $\hat{h}(t)=h(x)$
Taking into account that $\hat{h}$ is a single-well on $[0,1]$, then
in view of Theorem \ref{theo11}, we have $$
\frac{\tilde{\lambda}_{n}
}{\tilde{\lambda}_{m}}\leq(\frac{n}{m})^2.$$ Since
$\tilde{\lambda}_n=c^2\sigma^2(\lambda_n-\mu_1)$, then $
\frac{\lambda_{n} }{ \lambda_{m}}\leq(\frac{n}{m})^2.$ \\
 Assume that
there exist $q(x),$ $p(x)$ and $\rho(x)$ such that $
\frac{\lambda_{n}}{ \lambda_{m}}=(\frac{n}{m})^2$, where
$q(x)\not\equiv0$ or $p\rho(x)$ is not constant on $[0,1].$ It is
clear that the density in equation \eqref{1u1} is constant iff
$\hat{h}^4=\frac{\alpha}{ p\rho}$ on $[0,1]$ for some $\alpha>0$,
which is not possible from the monotonicity of $\hat{h}^4p\rho$.
This is in contradiction with Theorem \ref{theo11}. The proof of the
theorem is complete.\end{proof}
\\
we can now prove Corollary \ref{c3}.\\
\begin{proof}
Following the proof of Theorem \ref{theo13}, let $h$ be the solution
of Problem \eqref{0'}-\eqref{H1}. As before
$\min(\hat{\eta}_1,\tilde{\eta}_1)>0,$ we have $h(x)>0$ on $[0,1]$.
Since $q(x)\geq0$ on $[0,1]$, then by \eqref{0'} and \eqref{H1},
$h'(x)\leq0$ on $[0,\frac{1}{2}]$ and $h'(x)\geq0$ on
$[\frac{1}{2},1]$. Therefore $h$ is a single-well function on
$[0,1]$. The rest of the proof is similar to that of Theorem
\ref{theo13}.
\end{proof}
\begin{remark}\label{rem1}

 The method used in the proof of Theorem 1 cannot be applied in the
  case of single-barrier densities. More precisely, we have
  $$\frac{\lambda_n(\rho)}{\lambda_{m}(\rho)} \geq
\frac{\lambda_n(\hat{\rho})}{\lambda_{m}(\hat{\rho})}$$ on the other
hand,
$$\frac{\lambda_n(L)}{\lambda_{m}(L)}\leq(\frac{n}{m})^2.$$
I believe that different techniques are needed to deal with the case
of single barrier densities.
\end{remark}
~~~~~~~~~~~\\ {\bf{Acknowledgement}}.{ Research supported by Partial
differential equations laboratory $(LR03ES04),$ at the Faculty of
Sciences of Tunis, University of Tunis El Manar, $2092,$ Tunis,
Tunisia}

\vskip 3cm
\end{document}